\documentclass[12pt]{amsart}
\usepackage{hyperref}

\usepackage{bbm,stmaryrd,mathrsfs,mathabx,amssymb}
\usepackage[all]{xy}
\usepackage[pagewise,displaymath,mathlines]{lineno}

\newtheorem{theorem}{Theorem}[section]
\newtheorem{lemma}[theorem]{Lemma}
\newtheorem{prop}[theorem]{Proposition}
\newtheorem{corol}[theorem]{Corollary}

\theoremstyle{definition}
\newtheorem{defin}[theorem]{Definition}
\newtheorem{exam}[theorem]{Example}
\newtheorem{exams}[theorem]{Examples}

\theoremstyle{remark}
\newtheorem{remk}[theorem]{Remark}

\numberwithin{equation}{section}

\def\nid{\noindent}

\def\La{\Lambda}      \def\Th{\Theta}
\def\Ga{\Gamma}

\def\al{\alpha}       \def\de{\delta}
\def\be{\beta}        \def\eps{\varepsilon}
\def\ga{\gamma}

\def\si{\sigma}       
\def\ze{\zeta}        

 \def\mR{\mathbb R}

\def\aK{\mathbbm k}

 \def\gR{\mathfrak r}

 \def\dN{\mathfrak N}

 \def\dR{\mathfrak R}

\def\dM{\mathfrak M} 

 \def\kT{\mathcal T}
 \def\kU{\mathcal U}

 \def\fR{\mathbf r}

\def\sC{\mathsf C} 
\def\sD{\mathsf D} 
\def\sE{\mathsf E} \def\sR{\mathsf R}
\def\sF{\mathsf F} 
\def\sG{\mathsf G} \def\sT{\mathsf T}
 
\def\sI{\mathsf I} 
\def\sJ{\mathsf J} 
 
\def\sL{\mathsf L}

 \def\rT{\mathrm T}
 \def\rU{\mathrm U}

\def\cA{\mathscr A} 
\def\cB{\mathscr B} 
\def\cC{\mathscr C} 
\def\cD{\mathscr D} 
 \def\cR{\mathscr R}
 
 \def\cT{\mathscr T}
 
\def\cI{\mathscr I} 
\def\cJ{\mathscr J}

\def\qA{{\boldsymbol A}} 
\def\qB{{\boldsymbol B}} 
\def\qC{{\boldsymbol C}} \def\qP{{\boldsymbol P}}
 \def\qQ{{\boldsymbol Q}}
 \def\qR{{\boldsymbol R}}

\def\qH{{\boldsymbol H}}

\def\tA{\boldsymbol{\tilde{A}}}		\def\tH{\boldsymbol{\tilde{H}}}
\def\oA{\boldsymbol{\bar{A}}}		\def\oH{\boldsymbol{\bar{H}}}
\def\DF{\sD\sF}			\def\LF{\sL\sF}		\def\RF{\sR\sF}
\def\DI{\sD\sI}			\def\LI{\sL\sI}		\def\RI{\sR\sI}
\def\DG{\sD\sG}			\def\LG{\sL\sG}		\def\RG{\sR\sG}
\def\DJ{\sD\sJ}			\def\LJ{\sL\sJ}		\def\RJ{\sR\sJ}
\def\och{\bar{\chi}}		\def\tB{\boldsymbol{\tilde{B}}}

\def\mt{\mbox{-}}			\def\xx{\times}
\def\*{\otimes}		\def\+{\oplus}		
\def\bop{\bigoplus}	\def\cc{\boldsymbol{\cdot}}
\def\sb{\subset}         
\def\spe{\supseteq}      \def\sbe{\subseteq}
\def\0{\emptyset}		\def\8{\infty}
\def\ito{\stackrel{\sim}{\to}}
\def\xarr{\xrightarrow}
\def\ch{^{\scriptscriptstyle\vee}}
\def\<{\langle}		\def\>{\rangle}
\def\pto{\stackrel{\, \scriptscriptstyle+}{\to}}
\def\dagg{\mathop\dagger\nolimits}

\def\mtr#1{\begin{pmatrix}#1\end{pmatrix}}
\def\smtr#1{\left(\begin{smallmatrix}#1\end{smallmatrix}\right)}
\def\bmd#1#2#3{\left(\begin{smallmatrix}#1\\#2\end{smallmatrix}\ #3\right)}

\def\Hom{\mathop\mathrm{Hom}\nolimits}
\def\RHom{\mathop{\mR\mathrm{Hom}}\nolimits}
\def\Ext{\mathop\mathrm{Ext}\nolimits}
\def\Tor{\mathop\mathrm{Tor}\nolimits}
\def\gdim{\mathop\mathrm{l.gl.dim}\nolimits}
\def\pdim{\mathop\mathrm{pr.dim}\nolimits}
\def\ddim{\mathop\mathrm{der.dim}\nolimits}
\def\wdim{\mathop\mathrm{w.dim}\nolimits}
\def\Tri{\mathop\mathrm{Tri}\nolimits}
\def\Ker{\mathop\mathrm{Ker}\nolimits}
\def\Ob{\mathop\mathrm{Ob}\nolimits}
\def\Im{\mathop\mathrm{Im}\nolimits}
\def\End{\mathop\mathrm{End}\nolimits}
\def\ann{\mathop\mathrm{ann}\nolimits}
\def\rad{\mathop\mathrm{rad}\nolimits}
\def\Ver{\mathop\mathrm{Ver}\nolimits}
\def\Per{\mathop\mathrm{Perf}\nolimits}
\def\El{\mathop\mathsf{El}\nolimits}
\def\Cok{\mathop\mathrm{Cok}\nolimits}

\def\Md{\mbox{-}\mathrm{Mod}}
\def\md{\mbox{-}\mathrm{mod}}
\def\Id{\mathrm{Id}}
\def\op{^\mathrm{op}}

\def\bac{Backstr\"om }			\def\pbac{piecewise Backstr\"om}
\def\sor{semi-orthogonal decomposition}
\def\iff{if and only if }

\begin{document}

\title{Backstr\"om algebras}
\author{Yuriy Drozd}
\thanks{This work was accomplished during the visits of the author to the Max-Plank-Institut of Mathematics and the University of Paderborn, and he is grateful
to these institutions for their hospitality and financial support.}
\address{Institute of Mathematics, National Academy of Sciences of Ukraine, Tereschenkivska str. 3, 01024 Kyiv, Ukraine}
\email{y.a.drozd@gmail.com, drozd@imath.kiev.ua}
\urladdr{www.imath.kiev.ua/~drozd}

\begin{abstract}
 We introduce \emph{\bac pairs} and \emph{\bac rings}, study their derived categories and construct for them a sort of \emph{categorical resolutions}. For the 
 latter we define the global dimension, construct a sort of \sor{s} of the derived category and deduce that the derived dimension of a \bac ring is at most $2$. 
 Using this \sor, we define a description of the derived category as the category of elements of a special bimodule.
 We also construct a partial tilting for a \bac pair to a ring of triangular matrices and define the global dimension of the latter.
\end{abstract}

\maketitle
\tableofcontents

 \section*{Introduction}

 \bac orders were introduced in \cite{rinrog}, where it was shown that their representations are in correspondence with those of quivers or species. 
 A special class of \bac orders are \emph{nodal orders}, which appeared (without this name) in \cite{pure} as such pure noetherian algebras that the 
 classification of their finitely generated modules is tame. In \cite{BD-nodal} tameness was also proved for the derived categories of nodal orders. 
 Global analogues of nodal algebras, called \emph{nodal curves}, were considered in \cite{BD-tilt,DVol1,DVol2}. Namely, in \cite{BD-tilt} a sort of tilting 
 theory for such curves was developed, which related them to some quasihereditary finite dimensional algebras. In \cite{DVol1} a criterion was found 
 for a nodal curve to be tame with respect to the classification of vector bundles, and in \cite{DVol2} it was proved that the same class of curves 
 has tame derived categories. It was clear that the tilting theory of \cite{BD-tilt} can be extended to a general situation, namely, to \emph{\bac curves}, 
 i.e. non-commutative curves having \bac orders as their localizations. Nodal orders and related gentle algebras appear in studying mirror symmetry, 
 see, for instance, \cite{LekPol}. Finite dimensional analogues of nodal orders, called \emph{nodal algebras}, were introduced in \cite{dz1,Zembyk}. 
 In the latter paper their structure was completely described. In \cite{Zembyk2} it was shown that such important classes as gentle and 
 skewed-gentle algebras are nodal. In \cite{arXiv-gentle} a tilting theory was developed for nodal algebras, which was applied to the study 
 of derived categories of gentle and skewed-gentle algebras.
 
 This paper is devoted to a tilting theory for \emph{\bac rings}, which are a straightforward generalization of \bac orders and algebras. In Section~1 we 
 propose a variant of partial tilting, which generalizes the technique of minors from \cite{BDG}. In Section~2 we introduce \emph{\bac pairs} as the pairs 
 of semiperfect rings $\qH \spe \qA$ with common radical and \emph{\bac rings} (\emph{\pbac\ rings}) as the rings $\qA$ that occur in \bac pairs with 
 hereditary (respectively, picewise hereditary) $\qH$. 
 We construct the \emph{Auslander envelope} $\tA$ of a \bac pair and calculate its global dimension. Actually, this global 
 dimension only depends on the global dimension of $\qH$. In particular, Auslander envelopes for \bac rings are of global dimension at most 2. 
 In Section~3 we apply the tilting technique to show that the derived category of the algebra $\qA$ is connected by a recollement with the derived category 
 of its Auslander envelope. It implies that the derived dimension of $\qA$ in the sense of Rouquier \cite{Rouq} is not greater than that of the Auslander 
 envelope. In Section~4 we consider a recollement between the derived categories of the algebra $\qH$ and of the Auslander envelope. It is used 
 to calculate the derived dimension of the Auslander envelope, thus obtaining an upper bound for the derived dimension of the algebra $\qA$. 
 In particular, we prove that the derived dimension of a \bac or \pbac\ algebra is at most $2$. Moreover, if $\qA$ is a \bac or \pbac\ algebra of Dynkin type,
 then either it is piecewise hereditary of Dynkin type, so $\ddim\qA=0$, or $\ddim\qA=1$. In Section~5 we establish an equivalence between the category 
 $\cD(\tA)$ and a bimodule category. Such equivalence gives a useful instrument for calculations in this derived category (see, for instance, 
 \cite{BekMerk2,BekMerk1,BD-nodal,BDnodproj,arXiv-gentle,DVol2}). 
 In Section 6~we consider another partial tilting for the Auslander envelope $\tA$ of a \bac pair, relating its derived category by a recollement to the derived 
 category of an algebra $\qB$ of triangular matrices which looks simpler than the Auslander algebra. In this case we calculate explicitly the global
 dimension of $\qB$ and the kernel of the partial tilting functor $\sF : \cD(\qB) \to \cD(\qA)$.
 
 The author is thankful to the referee for important remarks that allowed to essentially improve the paper.

 \section{Partial tilting}
  \label{sec1} 
  
  Let $\cT$ be a triangulated category, $\dR\sbe \Ob\cT$. We denote by $\Tri(\dR)$ the smallest strictly full triangulated subcategory containing $\dR$ and closed under coproducts (it means that if a coproduct of objects from $\Tri(\dR)$ exists in $\cT$, it belongs to $\Tri(\dR)$). For a DG-category $\cR$ we denote by $\cD(\cR)$ its derived category \cite{Keller-DG}. The following result is a generalization of \cite[Prop.\,2.6]{Lunts}.

\begin{theorem}\label{t11} 
Let $\dR$ be a subset of the set of compact objects of $\Ob\cD(\cA)$, where $\cA$ is a Grothendieck category. We consider the DG-category $\cR$ 
with the set of objects $\dR$ and the sets of morphisms $\cR(T,R) = \RHom(T,R)$. Define the functor $\sF : \cD(\cA)\to \cD(\cR\op)$ mapping a complex $C$ 
to the DG-module $\sF C = \RHom_{\cD(\cA)}(- ,C)$ restricted onto $\dR$.

\begin{enumerate}
\item The restriction of $\sF$ onto $\Tri(\dR)$ is an equivalence $\Tri(\dR)\to\cD(\cR\op)$.

\item There is a recollement diagram in the sense of \cite[1.4.3]{BBD}.\,%
 \footnote{\,Note that $\dR$ is not necessarily \emph{recollement-defining} in the sense of \cite{NS}.}
\begin{equation}\label{e11} 
  \xymatrix{ \Ker\sF \ar[rr]|{\,\sI\,} && \cD(\cA) \ar@/_1em/[ll]|{\,\sI^*\,}  \ar@/^1em/[ll]|{\,\sI^!\,}  \ar[rr]|{\,\sF\,}  &&
  	\cD(\cR\op)  \ar@/_1em/[ll]|{\,\sF^*\,}  \ar@/^1em/[ll]|{\,\sF^!\,}, }
\end{equation}
 where $\sI$ is the embedding.
  \\
  \emph{Recall that it means that the following conditions hold:}
     \begin{enumerate}
	 \item  $F$ and $I$ are exact. 
 	 \item  $\sF\sI=0$.
	 \item  $F^*$ and $F^!$ are, respectively, a left and a right adjoint functors to $F$.
	 \item  Both adjunction morphisms $\eta :\Id_{\cD(\cR\op)}\to \sF\sF^*$ and $\ze : \sF\sF^!\to \Id_{\cD(\cR\op)}$ are isomorphisms.
 	\item  The same holds for the triple $(\sI,\sI^*,\sI^!$).
 \end{enumerate}
 \emph{Note that the condition 1.4.3.4 from \cite{BBD} is a consequence of the other ones, see \cite[9.2]{Neeman}.}
\end{enumerate}
\medskip\noindent
\emph{If $\dR$ generates $\cD(\cA)$, we obtain an equivalence $\cD(\cA)\simeq\cD(\cR\op)$, as in \cite{Lunts}. If $\dR$ consists of one object $R$, 
we obtain an equivalence $\Tri(R)\simeq\cD(\qR\op)$, where $\qR = \RHom(R, R)$.}
\end{theorem}

 \begin{proof}
 (1) We identify $\cD(\cA)$ with the homotopy category $\cI(\cA)$ of \emph{K-injective complexes}, i.e. such complexes $I$ that $\Hom(C,I)$ 
 is acyclic for every acyclic complex $C$, and suppose that $\dR \sbe \cI (\cA)$. Then $\RHom$ coincides with $\Hom$ within the category $\cI(\cA)$ 
 so, for $C \in \cI(\cA)$, $\sF C = \Hom_{\cI(\cA)}(-,C))$ restricted onto $\dR$. The full subcategory of $\cI (\cA)$ consisting of such complexes $C$ 
 that the natural map $\Hom_{\cI (\cA)}(R, C) \to\Hom_{\cD(\cR\op)}(\sF R, \sF C)$ is bijective for all $R \in\dR$ contains $\dR$, is strictly full, triangulated 
 and closed under coproducts, since all objects from $\cR$ are compact. Therefore, it contains $\Tri(\dR)$. Quite analogously, the full subcategory 
 of such complexes $C$ that the natural map $\Hom_{\cI (\cA)}(C, C')\to \Hom{\cI (\cA)}(\sF C, \sF C')$ is bijective for every $C'\in\Tri(\dR)$ 
 also contains $\Tri(\dR)$. Hence the restriction of $\sF$ onto $\Tri(\dR)$ is fully faithful. Moreover, as the functors $\Hom_\cR(- ,R)$, where $R$ 
 runs through $\dR$, generate $\cD(\cR\op)$, the functor $\sF$ is essentially surjective. Therefore, restricted to $\Tri(\dR)$, it gives an equivalence 
 $\Tri(\dR)\to \cD(\cR)$.

(2) Note that $\cD(\cR\op)$ is cocomplete and compactly generated, hence satisfies the Brown representability theorem \cite[Thm.\,8.3.3]{Neeman}. 
 Therefore, it is true for $\Tri(\dR)$ too. Then \cite[Prop.\,9.1.19]{Neeman} implies that a Bowsfield localization functor exists for $\Tri(\dR)\sbe\cD(\cA)$ 
 and \cite[Prop.\,9.1.18]{Neeman} implies that the embedding $\sE :\Tri(\dR)\to \cD(\cA)$ has a right adjoint $\Th : \cD(\cA)\to\Tri(\dR)$. 
 Let $\sF' :\cD(\cR\op)\to\Tri(\dR)$ be a quasi-inverse to the restriction of $\sF$ onto $\Tri(\dR)$. In particular, $\sF'$ is a left adjoint to this restriction and 
 the adjunction $\sF\sF'\to \Id_{\cD(\cR\op)}$ is an isomorphism. Then 
 $\sF C = \Hom_{\cI (\cA)}(- , C)|_\dR \simeq \Hom_{\cI (\cA)}(-, \Th C)|_\dR = \sF\Th C.$ 
 Set $\sF^* = \sE\sF'$. Since $\sF'M \in\Tri(\dR)$ for every $M \in\cD(\cR\op)$,
 \begin{equation}\begin{split}
   \Hom_{\cI (\cA)}(\sF^*M, C)&\simeq  \Hom_{\Tri(\dR)}(\sF'M, \Th C) \simeq \\
		&\simeq \Hom_{\cD(\cR\op)}(M, \sF\Th C) \simeq \Hom_{\cD(\cR\op)}(M,\sF C),
\end{split} \notag\end{equation}
for any $M \in \cD(\cR\op)$ and $C \in\sI(\cA)$. Hence $\sF^*$ is a left adjoint to $\sF$. If, moreover, $C \in \Tri(\dR)$, we obtain
\[
 \Hom_{\cD(\cR\op)}(\sF\sF^*M, \sF C) \simeq \Hom_{\cI (\cA)}(\sF^*M, C) \simeq \Hom_{\cD(\cR\op)}(M, \sF C).
\]

\nid
As $\sF$ is essentially surjective, it implies that $\eta: \sF\sF^* \to \Id_{\cD(\cR\op)}$ is an isomorphism. As all objects from $\dR$ are compact, 
$\sF$ respects coproducts, hence has a right adjoint $\sF^!$ \cite[Thm.\,8.4.4]{Neeman}. Now it follows from 
\cite[Cor.\,2.3]{BDG} that $\ze:\sF\sF^!\to \Id_{\cD(\cR\op)}$ is an isomorphism and there is a recollement diagram \eqref{e11}.

 \end{proof}
 
 Note that $\Im \sF^* = \Tri(\dR)$ by construction, but usually $\Im \sF^! \ne \Tri(\dR)$, though it is equivalent to $\Tri(\dR)$.
 \begin{corol}\label{c12} 
  Under conditions and notations of the preceding theorem, suppose that $\Hom_{\cD(\cA)}(R,T[m]) = 0$ for $R,T \in\dR$ and $m\ne 0$. Then
the functor $\sF$ induces an equivalence $\Tri(R)\ito \cD(\cR\op)$, where $\cR$ is the category with the set of objects $\dR$ and the sets of morphisms 
$\cR(A,B) = \Hom_{\cD(\cA)}(A,B)$. 

\emph{In this situation we call the functor $\sF$ a \emph{partial tilting functor.}}
 \end{corol}

\section{\bac pairs}
\label{sec2} 

Recall \cite{Bass-findim,Lambek} that a \emph{semiperfect ring} is a ring $\qA$ such that $\qA/ \rad \qA$ is a semisimple artinian ring and idempotents 
can be lifted modulo $\rad\qA$. Equivalently, as a left (or as a right) $\qA$-module, $\qA$ decomposes into a direct sum of modules with local 
endomorphism rings.
 
 \begin{defin}\label{d21} 
  \begin{enumerate}
  \item  A \emph{\bac pair} is a pair of semiperfect rings $\qH\spe \qA$ such that $\rad\qA = \rad\qH$. We denote by $\qC(\qH,\qA)$ the \emph{conductor}
   of $\qH$ in $\qA$:
   \[
   \qC(\qH,\qA) = \{\al \in \qA \mid \qH\al\sbe \qA\}= \ann(\qH/\qA)_\qA
   \]   
   (the right subscript ${_\qA}$ means that we consider $\qH/\qA$ as a right $\qA$-module). Obviously, $\qC(\qH, \qA)\spe \rad \qA$, so both $\qA/\qC$ and
   $\qH/\qC$ are semisimple rings.

\item We call a ring $\qA$ a (left) \emph{\bac ring} (\emph{\pbac\ ring}) if there is a \bac pair $\qH\spe \qA$, where the ring $\qH$ is left hereditary
(respectively, $\qH$ is \emph{left piecewise hereditary} \cite{Happel}, i.e. derived equivalent to a left hereditary ring). If, moreover, both $\qA$ and $\qH$ 
are finite dimensional algebras over a field $\aK$, we call $\qA$ a \emph{\bac algebra} (respectively, \emph{\pbac\ algebra}).
  \end{enumerate}
 \end{defin}

\begin{remk}\label{r22} 
 If $e$ is an idempotent in $\qA$, then $\rad(e\qA e) = e(\rad \qA)e$, hence, if $\qH\spe\qA$ is a \bac pair, so is also $e\qH e\spe e\qA e$. 
It implies that if $P$ is a finitely generated projective $\qA$-module, $\qA' = \End_\qA P$ and $\qH' = \End_\qH(\qH\*_\qA P)$, then $\qH'\spe \qA'$ 
is also a \bac pair. Note that if $\qH$ is left hereditary (or piecewise hereditary), so is $\qH'$, hence $\qA'$ is a  \bac ring (\pbac\ ring) if so is $\qA$. 
In particular, the notion of  \bac (or \pbac) ring is Morita invariant. Note also that if $\qH$ is left hereditary and noetherian, it is also right hereditary, 
so $\qA\op$ is also a \bac ring (\pbac\ ring). 
\end{remk}

\begin{exams}\label{ex23} 
 \begin{enumerate}
 \item  An important example of \bac  algebras are \emph{nodal algebras} introduced in \cite{dz1,Zembyk}. By definition, they are finite dimensional  
 algebras such that there is a \bac  pair $\qH\spe\qA$, where $\qH$ is a hereditary algebra and $\mathrm{length}_\qA(\qH \*_\qA U)\le 2$ for every 
 simple $\qA$-module $U$. Their structure was completely described in \cite{Zembyk}.

 \item   Recall that a $\aK$-algebra $\qA$ is called \emph{gentle} \cite{AsSk} if $\qA \simeq \aK\Ga/J$, where $\Ga$ is a finite quiver (oriented graph) and $J$ 
 is an ideal in the path algebra $\aK\Ga$ such that $(J_+)^2 \spe J \spe (J_+)^k$ for some $k$, where $J_+$ is the ideal generated by all arrows, and the 
 following conditions hold:
 		\begin{enumerate}
 		\item   For every vertex $i \in \Ver \Ga$ there are at most 2 arrows starting at $i$ and at most 2 arrows ending at $i$.
 		\item   If an arrow $a$ starts at $i$ (ends at $i$) and arrows $b_1,b_2$ end at $i$ (respectively, start at $i$), then either $ab_1 = 0$ or $ab_2 = 0$ 
 		(respectively, either $b_1a = 0$ or $b_2a = 0$), but not both.
 		\item   The ideal $J$ is generated by products of arrows of the sort $ab$. 
 		\end{enumerate}
  It is proved in \cite{Zembyk2} that such algebras are nodal, hence \bac  algebras. The same is true for skewed-gentle algebras \cite{GdP} 
  obtained from gentle algebras by blowing up some vertices.
  
  \item   \emph{\bac  orders} are orders $\qA$ over a discrete valuation ring such that there is a \bac  pair $\qH \spe \qA$, where $\qH$ is a hereditary order. 
  They were considered in \cite{rinrog}. 
  
  \item    Let $\qH=\rT(n,\aK)$ be the ring of upper triangular $n\xx n$ matrices over a field $\aK$ and $\qA=\rU\rT(n,\aK)$ be its subring of unitriangular
  matrices $M$, i.e. such that all diagonal elements of $M$ are equal. Then $\qH$ is hereditary and $\rad\qH=\rad\qA$, hence $\qA$
  is a \bac algebra. In this case $\qC(\qH,\qA)=\rad\qA$.
  
  \item  $\La_n = \aK[x_1,x_2,...,x_n]/(x_1,x_2,...,x_n)^2$ embeds into $\qH = \prod_{i=1}^n \aK\Ga_i$, where $\Ga_i = \cc\xarr{a_i}\cc$ ($x_i$ 
  maps to $a_i$). Obviously, under this embedding $\rad\La_n = \rad\qH$, so $\La_n$ is a \bac algebra.
 \end{enumerate}
\end{exams}

We consider a fixed \bac  pair $\qH \spe \qA$, set $\gR = \rad\qA = \rad\qH$ and denote by $\qC$ the conductor $\qC(\qH,\qA)$. Obviously, $\qC$ 
is a two-sided $\qA$-ideal and the biggest left $\qH$-ideal contained in $\qA$. Actually, it even turns out to be a two-sided $\qH$-ideal and its definition is 
left-right symmetric.

\begin{lemma}\label{l24} 
  Let $R\sbe S$ be semisimple rings, $I=\{\al\in R\mid S\al\sbe R\}$. Then $I$ is a two-sided $S$-ideal.
\end{lemma}
\begin{proof}
Obviously, $I$ is a left $S$-ideal and a two-sided $R$-ideal. As $R$ is semisimple, $I = Re$ for some central idempotent $e \in R$. Then $Se \sbe Re$, 
so $Se = Re=eR$ and $(1-e)Se = 0$. Hence $eS(1-e)$ is a left ideal in $S$ and $(eS(1-e))^2 =0$, so $eS(1-e)=0$ and $I=Se=eS$ is also 
a right $S$-ideal.
\end{proof}

\begin{prop}\label{p25} 
  $\qC$ is a two-sided $\qH$-ideal. It is the biggest $\qH$-ideal contained in $\qA$. Therefore, it coincides with the set 
  $\{\al \in \qA \mid \al\qH \sbe \qA\}$ or with $\ann_\qA(\qH/\qA)$  considered as left $\qA$-module.
\end{prop}
\begin{proof}
 It follows from the preceding lemma applied to the rings $\qA/ \rad \qA$ and $\qH/\rad\qH$. 
\end{proof}
 In what follows we assume that $\qA \ne \qH$, so $\qC\ne \qA$. To calculate $\qC$, we consider a decomposition $\qA = \bop_{i=1}^m  A_i$, 
 where $ A_i$ are indecomposable projective left $\qA$-modules. Arrange them so that $\qH A_i \ne  A_i$ for $1\le i\le r$ and $\qH A_i = A_i$ 
 for $r<i\le m$, and set $A^0 =\bop_{i=1}^rA_i,\,H^0 =\qH A^0$ and $A^1 =\bop_{i=r+1}^mA_i =\qH A^1$. Then $\qA=A^0\+A^1$ and
 $\qH=H^0\+A^1$ (possibly, $r=m$, so $A^0 =\qA$ and $H^0 =\qH$). Let $A^0 =\qA e_0$ and $A^1 =\qA e_1$, where $e_0$ and $e_1$ are 
 orthogonal idempotents and $e_0 + e_1 = 1$. Set $A^a_b = e_b\qA e_a$ and $H_b^a = e_b\qH e_a$, where $a,b \in \{0,1\}$. 
 Note that $A^1_b = H_b^1$ and $A^0_1 = H_1^0$. As $A^0$ and $A^1$ have no isomorphic direct summands, $A^a_b \sbe \rad \qA$ if 
 $a \ne b$. Hence, if we set $\gR^a = \rad A^a\ (a = 0, 1)$ and consider the Pierce decomposition of the ring $\qA$:
 
 \begin{align*}
  \qA&=\mtr{ A^0_0 & A^1_0 \\ A^0_1 & A^1_1}, \\
  \intertext{the Pierce decomposition of the ideal $\gR$ becomes}
  \gR&=\mtr{ \gR^0_0 & A^1_0 \\ A^0_1 & \gR^1_1},
 \end{align*}
 where $\gR^a_a=\rad A^a_a\ (a=0,1)$.
 It implies, in particular, that $H^0$ and $H^1$ have no isomorphic direct summands, the Pierce decomposition of $\qH$ is
 
 \[
   \qH =  \mtr{ H^0_0 & A^1_0 \\ A^0_1 & A^1_1} 
 \]
 and $\gR^0_0=\rad H^0_0$.
 Now one easily sees that an element $a =\smtr{\al&\be\\\ga&\de}$ belongs to $\qC$ if and only if $H^0\al \sbe A^0$. We claim that then 
 $H^0\al \sbe \rad A^0$. Otherwise $H^0\al$ contains an idempotent, hence a direct summand of $A^0$, which is isomorphic to some $A_i$ 
 with $1\le i \le r$. It is impossible, since $\qH A_i \ne A_i$. Therefore, $\al \in \gR^0_0$ and we obtain the following result.

 \begin{prop}\label{p26} 
  The Pierce decomposition of the ideal $\qC$ is
  \[
      \qC= \mtr{ \gR^0_0 & A^1_0 \\ A^0_1 & A^1_1}.
  \]
 \end{prop}

 \begin{defin}\label{d27} 
   Analogously to \cite{BD-tilt}, we define the \emph{Auslander envelope} of the \bac pair $\qH\spe\qA$ as the ring $\tA$ of $2\xx2$ matrices of the form
   
 \[
  \tA=\mtr{\qA & \qH \\ \qC & \qH} 
 \]
 with the usual matrix multiplication.
 \end{defin}

 Using Pierce decompositions of $\qA,\qH$ and $\qC$, we also present $\tA$ as the ring of $4 \xx 4$ matrices
 
 \begin{equation}\label{e21} 
  \tA=\mtr{ A^0_0 & A^1_0 & H^0_0 & A^1_0 \\ A^0_1 & A^1_1 & A^0_1 & A^1_1 \\ 
  			\gR^0_0 & A^1_0 & H^0_0 & A^1_0 \\ A^0_1 & A^1_1 & A^0_1 & A^1_1 }
 \end{equation}
 
 We also define $\tH$ as the ring of $4\xx 4$ matrices of the form
 
 \begin{align*}
   \tH&= \mtr{\qH & \qH \\ \qC & \qH} 
     \intertext{or}	 
  \tH&= \mtr{ H^0_0 & A^1_0 & H^0_0 & A^1_0 \\ A^0_1 & A^1_1 & A^0_1 & A^1_ 1 \\
  				 \gR^0_0 & A^1_0 & H^0_0 & A^1_0 \\ A^0_1 & A^1_1 & A^0_1 & A^1_1 } \\ 			 
 \end{align*}
 Obviously, $\rad\tH= \rad\tA$, so $\tH\spe\tA$ is also a \bac pair. $\tA$ is left noetherian \iff $\qA$ is left noetherian and $\qH$ is finitely generated 
 as left  $\qA$-module.
 
In the noetherian case one can calculate the global dimensions of $\tA$ and $\tH$. It turns out that it only depends on $\qH$.
 
\begin{theorem}\label{t28} 
 Suppose that either $\qA$ (hence also $\qH$) is left perfect or $\qA$ is left noetherian and $\qH$ is finitely generated as left $\qA$-module
 (hence also left noetherian). 
 \begin{align*}
  \gdim\tA&= 1 + \max(1 + \pdim_\qH \gR^0, \pdim_\qH \gR^1)=\\
  			&=
  			\begin{cases}
  			 1 + \gdim \qH &\text{\rm if } \pdim_\qH \gR^0 \ge \pdim_\qH\gR^1,\\
  			\gdim \qH &\text{\rm if } \pdim_\qH \gR^0 < \pdim_\qH \gR^1
  			\end{cases}\\
  		\intertext{and}
  		   & \gdim \tH = \gdim \qH,
 \end{align*}
 where we set $\pdim0=-1$.	   
 In particular, if $\qA$ is a \bac ring, so is also $\tA$, and if $\qA$ is not left hereditary, then $\gdim \tA = 2$.%
 \footnote{\,Note that if $\tA$ is left hereditary, so is $\qA=e'\tA e'$ \cite{San}.)}
\end{theorem}

\noindent 
For instance, it is the case for nodal (in particular, gentle or skewed-gentle) algebras (Example~\ref{ex23}).
 
\begin{proof}
 Under these conditions $\tA$ and $\tH$ are either left perfect or left noetherian.
 We recall that if a ring $\La$ is left perfect or left noetherian and semiperfect, then $\gdim\La = \pdim_\La(\La/ \rad \La) = 1 + \pdim_\La \rad \La$. 
 The $4\xx 4$ matrix presentation \eqref{e21} of $\tA$ implies that the corresponding presentation of $\rad \tA$ is
 
 \begin{equation}\label{e22} 
  \rad\tA=\mtr{ \gR^0_0 & A^1_0 & H^0_0 & A^1_0 \\ A^0_1 & \gR^1_1 & A^0_1 & \gR^1_1 \\ 
  			\gR^0_0 & A^1_0 & \gR^0_0 & A^1_0 \\ A^0_1 & \gR^1_1 & A^0_1 & \gR^1_1 }
 \end{equation}
 An $\tA$-module $M$ is given by a quadruple $(M',M'',\phi,\psi)$, where $M'$ is an $\qA$-module, $M''$ is an $\qH$-module,
 $\psi : M''\to M'$ is a homomorphism of $\qA$-modules and $\phi: \qC\*_\qA M'\to M''$ is a homomorphism of $\qH$-modules. 
 Namely, $M'=e'M,\,M''=e''M$, where $e'=\smtr{1&0\\0&0},\,e''=\smtr{0&0\\0&1}$, $\psi(m'')=\smtr{0&1\\0&0}m''$
 and $\phi(c\*m')=\smtr{0&0\\c&0}m'$. 
 We usually write $M = \smtr{ M' \\ M''}$ not mentioning $\phi$ and $\psi$. For an $\qH$-module $N$ we define the $\tA$-module 
 $N^+ = \smtr{N\\N}$. Then $N\mapsto N^+$ is an exact functor mapping projective modules to projective ones, since 
 $\qH^+=\smtr{\qH\\\qH}$ is a projective $\tA$-module.
 
 We denote by $L^i$ and by $R^i$, respectively, the i-th column of the presentations \eqref{e21} and \eqref{e22}. 
 Then $R^1 = (\gR^0)^+$ and $R^2 = R^4 = (\gR^1)^+$, where $\gR^a = \gR e_a$. If
 
  \[
  \cdots\to F_k \to\cdots\to F_1\to F_0 \to N\to0
  \]
 is a minimal projective resolution of an $\qH$-module $N$, 
 
 \[
   \cdots\to F^+_k \to\cdots\to F^+_1\to F^+_0 \to N^+\to0
 \]
is a minimal projective resolution of $N^+$, so $\pdim_{\tA} N^+ = \pdim_\qH N$. In particular, 
$\pdim_{\tA} R^1 = \pdim_\qH \gR^0$ and $\pdim_{\tA} R^2 = \pdim_\qH\gR^1.$
For the module $R^3$ we have an exact sequence

\begin{equation}\label{e23} 
 0\to (\gR^0)^+ \to R^3 \to \smtr{H^0/\gR^0 \\ 0} \to 0.
\end{equation}
 Note that $H^0/\gR^0$ is a semisimple $\qA$-module  and $e_1(H^0/\gR^0) = 0$, hence it contains the same simple direct summands 
 as $A^0/\gR^0$. The same is true for $\smtr{H^0/\gR^0 \\ 0}$ and $\smtr{A^0/\gR^0 \\ 0}=L^1/R^1$. Hence
 \[
   \pdim_{\tA}\smtr{H^0/\gR^0 \\ 0}=1+\pdim_{\tA} R^1 =1+\pdim_\qH \gR^0.
 \]
Therefore, the exact sequence \eqref{e23} shows that $\pdim_{\tA}R^3 = 1+\pdim_\qH \gR^0$ and
\[
\pdim_{\tA} \rad\tA = \max(1 + \pdim_\qH \gR^0, \pdim_\qH \gR^1),
\]
which gives the necessary result for $\tA$.
On the other hand, $R^3$ is a projective $\tH$-module, whence $\gdim\tH =\gdim\qH$. 
 \end{proof}

\section{The structure of derived categories}
\label{sec3} 

In what follows we denote by $\cD(\qA)$ the derived category $\cD(\qA\Md)$. We denote by $\cD_f (\qA)$ the full subcategory of $\cD (\qA)$ 
consisting of complexes quasi-isomorphic to complexes of finitely generated projective modules. If $\qA$ is left noetherian, it coincides with the 
derived category of the category $\qA\md$ of finitely generated $\qA$-modules. We also use the usual superscripts $+,-,b$. By $\Per(\qA)$ 
we denote the full subcategory of perfect complexes from $\cD(\qA)$, i.e. complexes quasi-isomorphic to finite complexes of finitely generated 
projective modules. It coincides with the full subcategory of compact objects in $\cD(\qA)$ \cite{Rouq}. If $\qA$ is left noetherian, an $\qA$-module 
$M$ belongs to $\Per(A)$ if and only if it is finitely generated and of finite projective dimension.

There are close relations between the categories $\cD(\qA),\cD(\qH)$ and $\cD(\tA)$ based on the following construction \cite{BDG}.

Let $\qP =\smtr{\qA\\\qC}$. It is a projective $\tA$-module and $\End\qP\simeq \qA\op$, so it can be considered as a right $\qA$-module. 
Consider the functors

\begin{align*}
 & \sF=\Hom_{\tA}(\qP,- )\simeq\qP\ch\*_{\tA}- :\tA\Md\to\qA\Md,\\
 & \sF^*=\qP\*_\qA- :\qA\Md\to\tA\Md,\\
 & \sF^! = \Hom_\qA(\qP\ch,- ) : \qA\Md\to \tA\Md,
\end{align*}  
 where $\qP\ch =\Hom_{\tA}(P,\tA)\simeq(\qA\ \qH)$ is the dual right projective $\tA$-module. The functor $\sF$ is exact, $\sF^*$ is its left adjoint and 
 $\sF^!$ is its right adjoint. Moreover, the adjunction morphisms $\sF\sF^* \to \Id_{\qA\Md}$ and $\Id_{\qA\Md} \to \sF\sF^!$ are
isomorphisms \cite[Thm.\,4.3]{BDG}. The functors $\sF^*$ and $\sF^!$ are fully faithful and $\sF$ is essentially surjective, i.e. 
every $\qA$-module is isomorphic to $\sF M$ for some $\tA$-module $M$. $\Ker \sF$ is a Serre subcategory of $\tA\Md$ equivalent to $\oH\Md$,
where $\oH =\qH/\qC \simeq \tA/\smtr{\qA &\qH\\ \qC&\qC}$. The embedding functor $\sI : \Ker\sF\to \tA\Md$ 
has a left adjoint $\sI^*$ and a right adjoint $\sI^!$ and we obtain a recollement diagram
\[
     \xymatrix{ \Ker\sF \ar[rr]|{\,\sI\,} && \tA\Md \ar@/_1em/[ll]|{\,\sI^*\,}  \ar@/^1em/[ll]|{\,\sI^!\,}  \ar[rr]|{\,\sF\,}  && 
  	\qA\Md  \ar@/_1em/[ll]|{\,\sF^*\,}  \ar@/^1em/[ll]|{\,\sF^!\,} }
\]
As the functor $\sF$ is exact, it extends to the functor between the derived categories $\DF : \cD(\tA)\to \cD(\qA)$ acting on complexes componentwise. 
The derived functors $\LF^*$ and $\RF^!$ are, respectively, its left and right adjoints, the adjunction morphisms 
$\Id_{\cD(\qA)} \to\DF\cc\LF^*$ and $\DF\cc\LF^* \to\Id_{\cD(\qA)}$ are again isomorphisms and we have a recollement diagram
\[
     \xymatrix{ \Ker\DF \ar[rr]|{\,\DI\,} && \cD(\tA) \ar@/_1em/[ll]|{\,\LI^*\,}  \ar@/^1em/[ll]|{\,\RI^!\,}  \ar[rr]|{\,\DF\,}  && 
  	\cD(\qA)  \ar@/_1em/[ll]|{\,\LF^*\,}  \ar@/^1em/[ll]|{\,\RF^!\,} }
\]
 (It also follows from Corollary~\ref{c12}.) Here $\Ker \DF = \cD_{\oH} (\tA)$, the full subcategory of complexes whose cohomologies are 
 $\oH$-modules, i.e. are annihilated by the ideal $\smtr{\qA& \qH \\ \qC&\qC}$. Note that, as a rule, it is not equivalent to $\cD(\oH)$. 
 From the definition of $\sF$ it follows that 
 \[
   \Ker\DF = \qP^\perp = \{C \in \cD(\tA) \mid \Hom_{\cD(\tA)}(\qP,C[k]) = 0\, \text{ for all }\, k\}.
 \]
Obviously, $\DF$ maps $\cD^\si(\tA)$ to $D^\si(\qA)$ for $\si \in \{+,-,b\}$, $\LF^*$ maps $\cD^-(\qA)$ to $\cD^-(\tA)$ and $\RF^!$ maps 
 $\cD^+(\qA)$ to $\cD^+(\tA)$. If $\tA$ is left noetherian, $\DF$ maps $\cD_f(\tA)$ to $\cD_f(\qA$) and $\LF^*$ maps $\cD_f(\qA)$ to $\cD_f(\tA)$. 
 Finally, both $\DF$ and $\LF^*$ have right adjoints, hence map compact objects (i.e. perfect complexes) to compact ones. On the contrary, 
 usually $\LF^*$ does not map $\cD^b(\qA)$ to $\cD^b (\tA)$. For instance, it is definitely so if $\gdim \tA < \8$ while $\gdim\qA =\8$ as in 
 Example~\ref{ex23}\,(4,5). If $\gdim\qH$ is finite, so is $\gdim\tA$, thus this recollement can be considered as a sort of categorical resolution
 of the category $\cD(\qA)$. In any case, it is useful for studying the categories $\qA\Md$ and $\cD(\qA)$ if we know the structure of the categories 
 $\tA\Md$ and $\cD(\tA)$. For instance, it is so if we are interesting in the \emph{derived dimension}, i.e. the dimension of the category $\cD_f^b(\qA)$ 
 in the sence of Rouquier \cite{Rouq}. We recall the definition.
 
 \begin{defin}\label{d31} 
  Let $\cT$ be a triangular category and $\dM$ be a set of objects from $\cT$.
  \begin{enumerate}
  \item We denote by $\<\dM\>$ the smallest full subcategory of $\cT$ containing $\dM$ and closed under direct sums, direct summands and shifts 
  (not closed under cones, so not a triangulated subcategory).
  
  \item  If $\dN$ is another subset of $\cT$, we denote by $\dM\dagg\dN$ the set of objects $C$ from $\cT$ such that there is an exact triangle 
  $A \to B \to C \pto$, where $A \in\dM,\, B \in\dN$.
  
  \item  We define $\<\dM\>_k$ recursively, setting $\<\dM\>_1 = \<\dM\>$ and $\<\dM\>_{k+1} =\<\<\dM\>\dagg\<\dM\>_k\>$.
  
  \item The \emph{dimension} $\dim\cT$ of $\cT$ is the smallest $k$ such that there is a finite set of objects $\dM$ such that $\<\dM\>_{k+1} = \cT$ (if it exists). 
  We call the dimension $\dim\cD_f^b(\qA)$ the \emph{derived dimension} of the ring $\qA$ and denote it by $\ddim \qA$.
  \end{enumerate}
 \end{defin}

As the functor $\sF$ is exact and essentially surjective, the next result is evident.
\begin{prop}\label{p32} 
 $\ddim\qA\le \ddim\tA$. Namely, if $\cD_f^b(\tA) = \<\dM\>_{k+1}$, then $\cD_f^b(\qA) = \<\DF(\dM)\>_{k+1}$.
 
\end{prop}

\section{Semi-orthogonal decomposition}
\label{sec4} 

There is another recollement diagram for $\cD(\tA)$ related to the projective module $\qQ =\smtr{\qH\\\qH}$ with $\End \qQ\simeq \qH\op$. Namely, we set

\begin{align*}
 & \sG=\Hom_{\tA}(\qQ,- )\simeq \qQ\ch\*_{\tA}- :\tA\Md\to\qH\Md,\\
 & \sG^*=\qQ\*_\qH- :\qH\Md\to\tA\Md,\\
 & \sG^! = \Hom_\qH(\qQ\ch,- ) : \qH\Md\to\tA\Md,\\
  \intertext{ where  $\qQ\ch = \Hom_{\tA} (\qQ,\tA) \simeq (\qC\ \qH)$,}
  & \DG : \cD(\tA) \to \cD(\qH)\, \text{ is $\sG$ applied componentwise,}\\
  & \LG^* : \cD(A) \to \cD(\tA)\, \text{ is the left adjoint of }\, \DG,\\
  & \RG^! : \cD(A)\to \cD(\tA)\, \text{ is the right adjoint of }\, \DG.
\end{align*} 
 We also set $\oA  = \qA/\qC \simeq \tA / \smtr{\qC &\qH \\\qC &\qH}$. Then we have recollement diagrams
 \begin{align*}
  & \xymatrix{ \Ker\sG \ar[rr]|{\,\sJ\,} && \tA\Md \ar@/_1em/[ll]|{\,\sJ^*\,}  \ar@/^1em/[ll]|{\,\sJ^!\,}  \ar[rr]|{\,\sG\,}  && 
  	\qH\Md  \ar@/_1em/[ll]|{\,\sG^*\,}  \ar@/^1em/[ll]|{\,\sG^!\,} }\\
  	\intertext{and}
  &   \xymatrix{ \Ker\DG \ar[rr]|{\,\DJ\,} && \cD(\tA) \ar@/_1em/[ll]|{\,\LJ^*\,}  \ar@/^1em/[ll]|{\,\RJ^!\,}  \ar[rr]|{\,\DG\,}  && 
  	\cD(\qH)  \ar@/_1em/[ll]|{\,\LG^*\,}  \ar@/^1em/[ll]|{\,\RG^!\,}, }
 \end{align*}
 where $\Ker\sG\simeq\oA\Md$. Since the $\tA$-ideal $(\qC\ \qH)$ is projective as right $\tA$-module, \cite[Thm.\,4.6]{BDG} 
implies that $\Ker \DG \simeq \cD (\oA )$.

As usually, this recollement diagram gives \sor s \cite[Cor.\,2.6]{BDG}
 \begin{equation}\label{e41} 
  \cD(\tA ) = (\Ker \DG, \Im \LG^*) = (\Im \RG^!, \Ker \DG)
 \end{equation}
with $\Ker \DG \simeq \cD (\oA )$ and  $\Im \LG^* \simeq \Im \RG^! \simeq \cD(\qH)$ (though usually $\Im \LG^* \ne \Im \RG^!$).

 Recall \cite{KLunts} that a \emph{\sor} $\cT =(\cT_1, \cT_2)$, where $\cT_1, \cT_2$ are full traingulated subcategories of $\cT$, means that 
 $\Hom_\cT (T_2, T_1) = 0$ if $T_i \in \cT_i$ and for every object 
 $T \in\cT$ there is an exact triangle $T_1 \to T_2 \to T \pto$, where $T_i\in\kT_i$. Then the following fact holds.
 
 \begin{lemma}\label{l41}\hspace{-3pt}%
 \footnote{\,In \cite[Theorem~7.4]{Ps} this result is proved in the case when this decomposition arises from a recollement.}
  If $\cT = (\cT_1, \cT_2)$ is a \sor\ of a triangulated category $\cT$ , then
 \[
   \dim\cT \le \dim\cT_1 +\dim\cT_2 +1.
   \]
 \end{lemma}

 \begin{proof}
  First we show that for any subsets $\dM,\dN$ of objects of the category $\cT$  
  \begin{equation}\label{e42} 
    \begin{split}
      \<\dM\>_{k+1} \dagg\dN &\sbe \<\dM\>\dagg\< \<\dM\>_k \dagg N\>\sbe \\
      				&\sbe \underbrace{\<\dM\>\dagg\<\<\dM\>\dagg\<\<\dM\>\dagg\dots \<\<\dM\>}_{k+1}\dagg\dN\>\dots\>\>
    \end{split}
  \end{equation}
  Indeed, let $C \in \<\dM\>_{k+1} \dagg \dN$, i.e there is an exact triangle $A \to B \to C \pto$, where $A\in\<\dM\>_{k+1}, B\in\dN$. 
  There is also an exact triangle $A_1 \to A\to A_2 \pto$, where $A_1 \in \<\dM\>_k, A_2 \in \<\dM\>$. The octahedron axiom implies that there are
  exact triangles $A_1 \to B\to B'\pto$ and $A_2 \to B' \to C\pto$. Therefore, $B' \in \<\dM\>_k \dagg \dN$ and $C \in \<\dM\> \dagg \<\<\dM\>_k\dagg \dN \>$.
  
  Now, let $\<\dM\>_{k+1} = \cT_1$ and $\<\dN\>_{l+1} = \cT_2$. Then, for every $T \in\cT$ there is an exact triangle $T_1\to T_2\to T \pto$, where 
  $T_1\in \<\dM\>_{k+1}, T_2 \in\<\dN\>_{l+1}$. But, according to \eqref{e42}, $\<\dM\>_{k+1} \dagg \<\dN\>_{l+1}\sbe\<\dM\cup\dN\>_{k+l+2}$, so 
  $\cT =\<\dM\cup\dN\>_{k+l+2}$ and $\dim\cT \le k+l+1$.
 \end{proof}
 
 As $\oA$  is semisimple, any indecomposable object from $\cD(\oA)$ is just a shifted simple module, so $\cD_f^b(\oA) = \<\oA\>$ and $\ddim\oA = 0$. 
 If $\qH$ is hereditary, every indecomposable object from $\cD_f^b(\qH)$ is a shift of a module. For every module $M$ there is an exact sequence
 $0\to P' \to P \to M \to0$ with projective modules $P, P'$  and, since $\qH$ is semiperfect, every indecomposable projective $\qH$-module is a direct 
 summand of $\qH$. Hence $\cD_f^b(\qH) = \< \qH\>_2$ and $\ddim\qH\le1$.
 
 \begin{corol}\label{c42} 
   $\ddim \qA \le \ddim\qH + 1$. In particular, if $\qA$ is a \bac (or \pbac) ring, $\ddim\qA\le 2$.
 \end{corol}

A finite dimensional hereditary algebra is said to be of \emph{Dynkin type} if it has finitely many isomorphism classes of indecomposable 
modules. Such algebras, up to Morita equivalence, correspond to Dynkin diagrams \cite{DR,Gab1}. If the derived category of an algebra $\qH$ 
is equivalent to the derived category of a hereditary algebra of Dynkin type, we say that $\qH$ is \emph{piecewise hereditary of Dynkin type}.\!%
 \footnote{\,It is proved in \cite{Happel} that picewise hereditary algebras of Dynkin type are just iterated tilted algebras of Dynkin type.}
 We say that a \bac (or \pbac) algebra $\qA$ is of \emph{Dynkin type} if there is a \bac pair $\qH \spe \qA$, where $\qH$ is a hereditary 
 (piecewise hereditary) algebra of Dynkin type. For instance, it is so if $\qA$ is a gentle or skewed-gentle algebra \cite{Zembyk2}, or the algebra
 $\rU\rT(n\aK)$ of unitriangular matrices over a field (Example~\ref{ex23}\,(4)), or the algebra $\La_n$ from Example~\ref{ex23}\,(5). 
 In this case $\cD_f^b(\qH) = \<M_1,M_2,\dots,M_m\>_1$, where $M_1,M_2,\dots,M_m$
  are all pairwise non-isomorphic indecomposable $\qH$-modules, so $\ddim\qH = 0$.

In \cite{derdim0} it was proved that $\ddim\qA = 0$ for a finite dimensional algebra $\qA$ if and only if $\qA$ is a piecewise hereditary algebra 
of Dynkin type.

\begin{corol}\label{c43} 
  If $\qA$ is a \bac (or \pbac) algebra of Dynkin type (for instance, gentle or skewed-gentle), but is not piecewise 
  hereditary of Dynkin type, then $\ddim\qA = 1$.
\end{corol}

\begin{exam}\label{x44} 
 The path algebra of the commutative quiver
 \[
   \xymatrix@R=1ex{ & 2 \ar[dr]^{\al_1} && 4 \\ 1 \ar[ur]^{\al_0} \ar[dr]_{\be_0} && 3 \ar[ur]^\ga \ar[dr]_{\ga'} && \al_1\al_0=\be_1\be_0\\ & 2' \ar[ur]_{\be_1} && 4' } 
 \]
 is a tilted (hence piecewise hereditary) algebra of type $\tilde{D}_5$. At the same time it is \bac algebra of type $A_4$. Namely, it is a skewed-gentle
 algebra obtained from the path algebra of the quiver $1\to 2\to 3\to 4$ by blowing up the vertices $2$ and $4$.%
 \footnote{\,See \cite{Zembyk} for the construction of blowing up and its relation to nodal algebras.}
\end{exam}

\section{Relation to bimodule categories}
\label{sec5}

In this section we explain how a \sor\ allows to apply to calculations in a triangulated category the technique of matrix problems, namely, of
bimodule categories, as in \cite{SP}.

Let $\cA$ and $\cB$ be additive categories, $\kU$ be an $\cA\mt\cB$-bimodule, i.e. a biadditive functor $\cA\op\xx\cB\to \mathbf{Ab}$. Recall \cite{SP} 
that the \emph{bimodule category} or the \emph{category of elements} of the bimodule $\kU$ is the category $\El(\kU)$ with the set of objects 
$\bigcup_{A\in\cA, B\in\cB} \kU(A, B)$ and the set of morphisms $u \to v$, where $u \in \kU(A, B), v \in \kU(A', B')$ being the set of pairs 
 $\{(\al,\be)\mid \al:A'\to A,\be:B\to B',u\al=\be v\}$. Here we write, as usually, $u\al$ and $\be v$, respectively, instead of $\kU(\al,1_B)u$ and 
 $\kU(1_{A'},\be)v$. Bimodule categories appear when there is a semi-orthogonal decomposition of a triangulated category.
 
 \begin{theorem}\label{t51} 
  Let $(\cA,\cB)$ be a \sor\ of a triangulated category $\cC$. Consider the $\cA\mt\cB$-bimodule $\kU$ such that 
  $\kU(A,B) =\Hom_\cC(A,B)\ (A \in \cA, B \in \cB)$. For every $f : A\to B$ fix a cone $C f$, that is an exact triangle 
  $A\xarr{f} B \xarr{f_1} Cf\xarr{f_2} A[1]$. The map $f\mapsto Cf$ induces an equivalence of categories $\sC : \El(\kU) \ito\cC/\cJ$, where $\cJ$ is the ideal of
  $\cC$ consisting of morphisms $\eta$ such that there are factorizations $\eta = \eta'\xi = \ze\eta''$, where the source of $\eta'$ is in $\cA$ and the target 
  of $\eta''$ is in $\cB$. Moreover, $\cJ^2 = 0$, so $\sC$ induces a bijection between isomorphism classes of objects from $\El(\kU)$ and from $\cC$.\!%
  \footnote{\,This theorem is actually a partial case of \cite[Thm.\,1.1]{SP}.}
 \end{theorem}
 \begin{proof}
  As $(\cA,\cB)$ is a semi-orthogonal decomposition of $\cC$, every object from $\cC$ occurs in an exact triangle $A\xarr f B\to C\pto$, where 
  $A \in \cA, B \in \cB$, so $f$ is an object from $\El(\kU)$ and $C \simeq Cf$. Let $f' : A' \to B'$ be another object of $\El(\kU)$ and 
  $(\al,\be) : f \to f'$ be a morphism from $\El(\kU)$. Fix a commutative diagram   
  
  \begin{equation}\label{e51} 
  \vcenter{ \xymatrix{ A \ar[r]^f \ar[d]_\al & B \ar[r]^{f_1} \ar[d]_\be & Cf \ar[r]^{f_2} \ar[d]_\ga & A[1] \ar[d]_{\al[1]} \\
  								A' \ar[r]^{f'}  & B' \ar[r]^{f'_1} \ & Cf' \ar[r]^{f'_2}  & A'[1]  }  }
  \end{equation}
 It exists, though is not unique. Let $\ga'$ be another morphism making the diagram \eqref{e51} commutative an set $\eta=\ga-\ga'$. Then $\eta f_1 = 0$, 
 hence $\eta$ factors through $f_2$, and $f_2'\eta = 0$, hence $\eta$ factors through $f_1'$. Thus $\eta\in\cJ$. On the other hand, if $\eta : Cf \to Cf'$ 
 is in $\cJ$, the decomposition $\eta=\eta'\xi$ implies that $\eta f_1 = \eta'\xi f_1 = 0$ and the decomposition $\eta=\ze\eta''$ implies that 
 $f'_2\eta = f'_2\ze\eta'' = 0$, hence the morphism $\ga'=\ga+\eta$ makes the diagram \eqref{e51} commutative. Therefore, the class $\sC(\al,\be)$ of 
 $\ga$ modulo $\cJ$ is uniquely defined, so the maps $f\mapsto Cf$ and $(\al,\be) \mapsto\sC(\al,\be)$ define a functor $\sC : \El(\kU)\to \cC/\cJ$.
 
Let now $\ga : Cf\to Cf'$ be any morphism. Then $f_2'\ga f_1 = 0$, so $\ga f_1 = f_1'\be$ for some $\be:B\to B'$. Hence there is a morphism $\al:A\to A'$ making
the diagram \eqref{e51} commutative, i.e. defining a morphism $(\al,\be) : f \to f'$ such that $\ga \equiv \sC(\al,\be)\pmod\cJ$ . If $(\al',\be')$ is another such 
morphism, $f_1' (\be-\be') = 0$, so $\be-\be'=f'\xi$ for some $\xi:B\to A$. But $\xi=0$, so $\be=\be'$. In the same way $\al=\al'$. Hence the functor $\sC$ 
is fully faithful. As we have already noticed, it is essentially surjective, therefore defines an equivalence $\El(\kU) \ito\cC/\cJ$.
The equality $\cJ^2 = 0$ follows immediately from the definition and the conditions of the theorem.
 \end{proof}
  
We apply Theorem~\ref{t51} to \bac pairs $\qH \sbe\qA$ such that $\qA$ is left noetherian and $\qH$ is left hereditary and finitely generated as 
left $\qA$-module. For instance, it is so in the case of \bac algebras or \bac orders. Then the ring $\tA$ is also noetherian and $\qC$ is projective as left $\qH$-module. 
According to \eqref{e41}, $(\Ker \DG, \Im \LG^*)$ is a semi-orthogonal decomposition of $\cD(\tA )$. Moreover, both $\sG$ and $\sG^*$ map finitely 
generated modules to finitely generated, so the same is valid if we consider their restrictions onto $\cD_f (\tA )$ and $\cD_f (\qH )$.
Note also that $\sG^*$ is exact, so $\sG^*$ can be applied to complexes componentwise. The $\tA$-module $\sG^*M$
 can be identified with the module of columns $M^+ = \smtr{M\\M}$ with the action of $\tA$ given by the matrix multiplication. It gives 
an equivalence of $\cD(\qH)$ with $\Im\LG^*$. As $\qH$ is left hereditary, every complex from $\cD(\qH)$ is equivalent to a direct sum of 
shifted modules (see \cite[Sec.\,2.5]{Keller}). On the other hand, $\Ker \DG \simeq \cD (\oA )$ and $\oA$ is semisimple, since $\qC\spe\fR$. Hence every complex 
from $\cD(\oA)$ is isomorphic to a direct sum of shifted simple $\oA$-modules, which are direct summands of $\oA$.
So, to calculate the bimodule $\kU$, we only have to calculate $\Ext^i_{\tA} (\oA ,M^+)$, where $M$ is an $\qH$-module. 
Note also that $\qC^+$ is a projective $\tA$-module, since $\qC$ is a projective $\qH$-module. Therefore, a projective resolution of $\oA$ is 
$0 \to \qC^+ \xarr{\eps} \qP\to \oA \to0$ and $\pdim_{\tA}\oA=1$. Hence we only have to calculate $\Hom_{\tA}(\oA , M^+ )$ and $\Ext^1_{\tA} (\oA, M^+)$.

\begin{theorem}\label{t52} 
 \begin{enumerate}
 \item $\Hom_{\tA}(\oA,M^+)\simeq\ann_M \qC =\{u\in M \mid \qC u=0\}$. 
 
 \item  $\Ext^1_{\tA} (\oA ,M^+) \simeq \Hom_\qH(\qC,M)/(M/\ann_M \qC)$, where the quotient $M/\ann_M \qC$ embeds into $\Hom_\qH(\qC,M)$ 
 if we map an element $u\in M$ to the homomorphism $\mu_u : c\mapsto cu$.
 \end{enumerate}
\end{theorem}
\begin{proof}
  (1)  $\Hom_{\tA}(\oA ,M^+)$ is identified with the set of homomorphisms $\phi: \qP \to M^+$ such that $\phi\eps=0$. A homomorphism 
 $\phi:\qP \to M^+$ is uniquely defined by an element $u\in M$ such that $\phi\smtr{1\\0}=\smtr{u\\0}$. Namely, $\phi\smtr{a\\c}=\smtr{au\\cu}$.
Obviously, $\phi\eps = 0$ if and only if $\qC u = 0$, i.e. $u \in \ann_M \qC$.

(2)  $\Ext^1_{\tA} (\oA,M^+) \simeq \Hom_{\tA}(C^+,M^+)/\Hom_{\tA}(\qP,M^+)\eps$. As the functor 
$\sG^*$ is fully faithful, $\Hom_{\tA}(C^+,M^+) \simeq \Hom_\qH(\qC,M)$. Namely, $\psi :\qC \to M$ induces $\psi^+ : \qC^+\to M^+$ 
mapping $\smtr{a\\b}$ to $\smtr{\psi(a)\\\psi(b)}$. Let $\phi : \qP \to M^+$ corresponds, as above, to an element $u \in M$. Then 
$\phi\eps\smtr{a\\c} =\smtr{au\\cu}$, so it equals $\mu_u$, and $\Hom_{\tA}(P,M^+)\eps$ is identified with $M/\ann_M\qC$ embedded
into $\Hom_\qH(\qC,M)$ as above.
\end{proof}

Actually, in our case an object $\sE$ from the category $\El(\kU)$ (therefore, also an object from $\cD^b(\tA)$) is given by the vertices and
solid arrows of a diagram
\vspace*{-.8em}
 \[
    \xymatrix{  A_n \ar@(ur,ul)@{.>}[]_{\al_n} \ar[d]^{\mu_n} \ar[dr]^{\eta_n} & A_{n+1} \ar@(ur,ul)@{.>}[]_{\al_{n+1}} \ar[d]^{\mu_{n+1}} \ar[dr]^{\eta_{n+1}} & 
   A_{n+2}\ar@(ur,ul)@{.>}[]_{\al_{n+2}}  \ar[d]^{\mu_{n+2}} \ar[dr]^{\eta_{n+3}} &  A_{n+3} \ar@(ur,ul)@{.>}[]_{\al_{n+3}} \ar[d]^{\mu_{n+3}}\hskip1em \cdots 
     \\ M_n \ar@{.>}[r]^{\be_n} \ar@{.>}@(dl,dr)[]_{\ga_n} & M_{n+1} \ar@{.>}[r]^{\be_{n+1}}  \ar@{.>}@(dl,dr)[]_{\ga_{n+1}}
     & M_{n+2} \ar@{.>}[r]^{\be_{n+2}} \ar@{.>}@(dl,dr)[]_{\ga_{n+2}} &M_{n+3}  \ar@{.>}@(dl,dr)[]_{\ga_{n+3}}\hskip1em \cdots}
  \]
 (of arbitrary length), where $A_i$ are $\oA$-modules, $M_i$ are $\qH$-modules, $\mu_i\in\Hom_{\tA}(A_i,M_i^+)$ and 
 $\eta_i\in\Ext^1_{\tA}(A_i,M_{i-1}^+)$. A morphism between $\sE$ and $\sE'$ is given by the dotted arrows, where
  \begin{align*}
   & \al_i\in\Hom_{\oA}(A_i,A'_i)\simeq\Hom_{\tA}(A_i,A'_i),\\
   & \ga_i\in\Hom_\qH(M_i,M'_i)\simeq\Hom_{\tA}(M_i^+,(M'_i)^+),\\
   & \be_i\in\Ext^1_{\qH}(M_i,M'_{i+1})\simeq\Ext^1_{\tA}(M^+_i,(M'_{i+1})^+).
  \end{align*} 
  These morphisms must satisfy the relations 
  \begin{align*}
   & \mu'_i\al_i=\ga_i\mu_i,\\
   & \eta'_i\al_i=\ga_{i+1}\eta_i+\be_i\mu_i.
  \end{align*}

\section{Partial tilting for \bac pairs}
\label{sec6} 

Let $\qH \sbe\qA$ be a \bac pair. Consider the ring $\qB$ of triangular matrices of the form
\[
  \qB=\mtr{\oA &\oH \\ 0&\qH }.
\] 
 Let $e_1=\smtr{1&0\\0&0}$ and $e_2=\smtr{0&0\\0&1}$, $B_1=\qB e_1$ and $B_2=\qB e_2$ be projective $\qB$-modules given by the first 
 and the second column of $\qB$, i.e.
 \[
   B_1=\mtr{\oA\\0},\qquad  B_2=\mtr{\oH\\\qH}.
 \]
A $\qB$-module $M$ is defined by a triple $\bmd{M_1}{M_2}{\chi_M}$, where $M_1=e_1M$ is an $\oA$-module, $M_2=e_2M$ is an $\qH$-module
and $\chi_M:M_2\to M_1$ is an $\qA$-homomorphism such that $\Ker\chi_M\spe\qC M_2$ (it is nesessary, since $\qC M_1=0$). 
Namely, $\chi_M$ is the multiplication by $\smtr{0&1\\0&0}$. We write an element $u\in M$ as a column $\smtr{u_1\\u_2}$,
where $u_1=e_1u,u_2=e_2u$. Then
 \[
  \mtr{a&b\\0&c} \mtr{u_1\\u_2}=\mtr{au_1+\chi_M(bu_2)\\cu_2}.
 \]
 A homomorphism $\al : M \to N$ is defined by two homomorphisms $ \al_1 : M_1 \to N_1$ and $\al_2 :M_2\to N_2$ such that 
 $\al_1\chi_M=\chi_N\al_2$. We write $\al=\smtr{\al_1\\\al_2}$.
 
 \begin{prop}\label{p61} 
 \quad $ \gdim\qB = \max(\gdim \qH, \wdim \oH_\qH  + 1).$ 
 \\[2pt]
  In particular, if $\qH$ is left hereditary and $\oH$  is not flat as right $\qH$-module, then $\gdim \qB = 2$.
 \end{prop}
 \begin{proof}
 \cite[Thm.\,5]{P-Roos} shows that $\gdim\qB\le n$ \iff 
 \[
   \gdim\qH \le n\ \text{ \rm and }\ \mR^n\!\Hom_{\oA}(\oH\*_\qH-,- ) = 0.
 \]
 As the ring $\oA$  is semisimple,
 \[
  \mR^n\!\Hom_{\oA}(\oH\*_\qH- ,-)=\Hom_{\oA }(\Tor^\qH_n(\oH ,-),-). 
 \]
  It implies the first assertion. The second is obvious, since $\Tor_1^{\qH}(\oH,-)=0$ \iff $\oH$ is flat as right $\qH$-module.
 \end{proof}
 
 We denote by $R$ the $\qB$-module given by the triple $\bmd{\qH/\qA}{\qH}{\pi}$, where $\pi : \qH \to \qH/\qA$ is the natural surjection.
 
 \begin{prop}\label{p62} 
  \begin{enumerate}
  \item  $\End_\qB R \simeq \qA\op$.
  \item  $\pdim_\qB R=1$.
  \item  $\Ext^1_\qB(R,R)=0$.
  \end{enumerate}
  \smallskip
  \emph{Recall that the conditions (2) and (3) mean that $R$ is a \emph{partial tilting} $\qB$-module.}
 \end{prop}
 \begin{proof}
 The minimal projective resolution of $R$ is 
 \[
   0\to B_1\xarr\eps B_2 \to R\to 0,
 \]
where $\eps$ is the embedding, which gives (2). Any endomorphism $\ga$ of $R$ induces a commutative diagram
\[
   \xymatrix{ B_1 \ar[r]^\eps \ar[d]_{\ga_1} & B_2 \ar[d]^{\ga_2}\\ B_1 \ar[r]^\eps & B_2  }
\]
 As $\End_\qB B_2 \simeq \qH\op$, $\ga_2$ is given by multiplication with an element $h \in \qH$ on the right. If there is a commutative 
 diagram as above, necessarily $h \in\qA$, which proves (1). Finally, a homomorphism $\al : B_1 \to R$ maps the generator $\smtr{1\\0}$ of
 $B_1$ to an element $\smtr{\bar h\\0} \in R$. If $h$ is a preimage of $\bar h$ in $\qH$, then $\al$ extends to the homomorphism
 $B_2\to R$  that maps the generator $\smtr{0\\1}$ of $B_2$ to $\smtr{0\\h}\in R$. It implies (3).
 \end{proof}
 
 Now Theorem~\ref{t11} applied to the module $R$ gives the following result.
 
 \begin{theorem}\label{t63} 
 \begin{enumerate}
 \item  The functor $\sF = \RHom(R,- )$ induces an equivalence $\Tri(R)\ito \cD(\qA)$.

 \item  $\Ker\sF$ consists of complexes $C$ such that the map $\chi_{H^k(C)}$ is bijective for all $k$.
 \item  There is a recollement diagram
 \[
  \xymatrix{ \Ker\sF \ar[rr]|{\,\sI\,} && \cD(\qB) \ar@/_1em/[ll]|{\,\sI^*\,}  \ar@/^1em/[ll]|{\,\sI^!\,}  \ar[rr]|{\,\sF\,}  && 
  	\cD(\qA)  \ar@/_1em/[ll]|{\,\sF^*\,}  \ar@/^1em/[ll]|{\,\sF^!\,} }
\]
 \end{enumerate}
 \smallskip
 \emph{Actually, the claim (2) means that a complex $C$ is in $\Ker \sF$ if and only if its cohomologies are direct sums of $\qB$-modules of the form
  $\bmd U U {1_U}$, where $U$ is a simple $\oH$-module.}
 \end{theorem}
 
 $\sF$ is a partial tilting functor in the sense of Corollary~\ref{c12}.
 
 \begin{proof}
 (1) and (3) follow from Proposition~\ref{p62} and Theorem~\ref{t11}, since the complex $P:\ 0\to B_1\xarr\eps B_2\to0$
  is perfect, hence compact, and isomorphic to $R$ in $\cD(\cB)$. 
 To find  $\Ker\sF$, consider a complex
 \[
  C:\ \cdot \to C^{k-1} \xarr{d^{k-1}} C^k \xarr{d^k} C^{k+1}\to\cdots
 \] 
 where $C^k$ is defined by a triple $\bmd{C^k_1}{C^k_2}{\chi_k}$ and $d^k =\smtr{d^k_1\\d^k_2}$, where $d^k_1\chi_k =\chi_{k+1}d^k_2$
 for all $k$. Note that $C_i = (C_i^k, d_i^k )\ (i = 1, 2)$ are complexes, $(\chi_k)$ is a homomorphism of complexes and 
 $H^k(C) =\bmd{H^k(C_1)}{H^k(C^2)}{\och_k}$, where $\och_k = \chi_{H^k(C)}$ is induced by $\chi_k$. A homomorphism $P \to C[k]$
 is a pair of homomorphisms $\al:B_2 \to C^k,\,\be:B_1\to C^{k-1}$ such that $\al_1\pi=\chi_k\al_2,\,\be_2 =0,\,d^k_i\al_i =0\ (i=1,2)$
 and $d^{k-1}\be_1=\al_1|_{\oA}$. Let $\al_2(1) = x \in C^k_2$ and $\be_1(1) = y \in C^{k-1}_1$. These values completely define 
 $\al$ and $\be$. The conditions for $\al$ and $\be$ mean that $d^k_2x = 0$ and $d^{k-1}y = \chi_kx$.
 
 This morphism is homotopic to zero if and only if there are maps $\si : B_2\to C^{k-1}$ and $\tau : B_1\to C^{k-2}$ such that 
 $\al = d^{k-1}\si$ and $\be = \si\eps+d^{k-2}\tau$. Again $\si$ is defined by the element $z = \si_2(1) \in C^{k-1}_2$ and $\tau$ is 
 defined by the element $t = \tau_1(1) \in C^{k-2}_1$. Then the conditions for $\al$ and $\be$ mean that $x=d_2^{k-1}z$ and 
 $y=\chi_{k-1}z+d_1^{k-2}t$.
 
 Suppose that any homomorphism $P \to C[k]$ is homotopic to zero. Let $\bar{x}  \in H^k(C^2)$ be such that $\och_k(\bar{x}) = 0$ and 
 $x \in \Ker d^k_2$ be a representative of $\bar{x}$. Then $\chi_k(x) = d^{k-1}_1y$ for some $y \in C^{k-1}$, so the pair $(x,y)$ defines a 
 homomorphism $P \to C[k]$. Therefore, there must be $z \in C^{k-1}_2$ such that $x=d^{k-1}z$, thus $\bar{x } = 0$ and $\och_k$ is injective. 
 Let now $\bar{y} \in H^{k-1}(C_2)$ and $y \in C^{k-1}_2$ be its representative. Then the pair $(0, y)$ defines a homomorphism 
 $P \to C[k]$, so there must be elements $z \in C^{k-1}_2$ and $t\in C^{k-2}_1$ such that $d_1^{k-1}z = 0$ and $y = \chi_{k-1}z + d_1^{k-2}t$. 
 Hence $\bar{y} = \och_{k-1}(\bar{z})$, so $\och_{k-1}$ is surjective. 
As it holds for all $k$, we have that all maps $\och_k$ are bijective.

 On the contrary, suppose that all $\och_k$ are bijective. If a pair $(x,y)$ defines a homomorphism $P \to C[k]$, then $\chi_k(x) = d_1^{k-1}y$, 
 so $\och_k(x) = 0$. Therefore, $\bar{x}=0$, i.e. $x=d^{k-1}_2z$ for some $z\in C^{k-1}_2$ and $\chi_kx=d^{k-1}_1\chi_{k-1}z$. Then
 $d^{k-1}_1(y-\chi_{k-1}z)=0$, hence there is an element $z' \in C^{k-1}_2$ such that $d_2^{k-1}z' = 0$ and the cohomology class of 
 $y - \chi_{k-1}z$ equals $\och_{k-1}\bar{z}'$, i.e. $y-\chi_{k-1}z=\chi_{k-1}z'+d_1^{k-2}t$ for some $t$. Then $x = d^{k-1}_2(z+z')$ and 
 $y=\chi_{k-1}(z+z') + d_1^{k-2}t$, so this homomorphism is homotopic to zero.
 \end{proof}
  
  As usually, we identify the category $\qA\Md$ with the full subcategory of $\cD(\qA)$ consisting of the complexes $C$ concentrated in degree 0. 
  The following result shows how the parftial titling functor $\sF$ behaves with respect to modules.
  
  \begin{corol}\label{c64} 
  Let a $\qB$-module $M$ be given by the triple $\bmd{M_1}{M_2}{\chi_M}$.
  \begin{enumerate}
  \item  $\sF M \in \qA\Md$ if and only if $\chi_M$ is surjective. Namely, then $\sF M\simeq\Ker\chi_M$.
  \item  $\sF M \in \qA\Md[1]$ if and only if $\chi_M$ is injective. Namely, then $\sF M\simeq\Cok\chi_M[1]$.
  \end{enumerate}
  \end{corol}
  \begin{proof}
  Note that $\Hom_\qB(B_1,M)\simeq M_1$, $\Hom_\qB(B_2,M)\simeq M_2$ and if $\phi:B_2\to M$ maps $\smtr{0\\1}$ to $\smtr{0\\x}$, then
  $\phi\eps$ maps $\smtr{1\\0}$ to $\smtr{\chi_M(x)\\0}$. Therefore, $\RHom_\qB(R,M)$ is the complex
  \[
    0\to M_2\xarr{\chi_M} M_1\to 0,  
  \]
  which proves the claim.
  \end{proof}
  
  \begin{remk}\label{r65} 
  There are several derived equivalences related to $\tA$.
  \begin{enumerate}
  \item If $\qA$ is a \bac order, it is known (see \cite{BDG}) that the complex $T=B_1[1]\+\qH^+$, where $B_1=\smtr{\oA\\0}$, is a tilting complex for $\tA$
  and $(\End_{\cD(\tA)})\op T\simeq\qB$, hence $\tA$ is derived equivalent to $\qB$. Nevertheless, in general situation of \bac rings 
  (even of \bac algebras) it is not so. First of all,  $\Hom_{\tA}(B_1,\qH^+)\simeq\ann_\qH\qC$, so it can happen that $\Hom_{\cD(\tA)}(T,T[1])\ne0$. 
  It is so, for instance, for the pair $\big(\rT(n,\aK),\rU\rT(n,\aK)\big)$ from Example~\ref{e23}\,(4), since in this case the matrix unit $e_{nn}$ belongs to
  $\ann_\qH\qC$. It is also so for Example~\ref{e23}\,(5).
  Moreover, even if $\ann_\qH\qC=0$, one can see that $\oH'=\Ext^1_{\tA}(B_1,\qH^+)\simeq \qC^{-1}/{_\qC\qH}$, where 
  $\qC^{-1}=\Hom_\qH(\qC,\qH)$ and ${_\qC\qH}=\qH/\ann_\qH\qC$ is naturally embedded into $\qC^{-1}$. Therefore, in this case
  \[
    (\End_{\cD(\tA)}T)\op \simeq\qB'=\mtr{\oA & \oH'\\ 0& \qH},
  \]
  which need not coincide with $\qB$ (see Example~\ref{e66} below). If $\qH$ is a hereditary order, then $\ann_\qH\qC=0$ and $\oH'\simeq\oH$, 
  hence $\qB'\simeq\qB$,  in accordance with \cite{BDG}.
 
 \medskip
 \item  On the other hand, set $T'=\smtr{\qA\ &\qH/\qA\\ \qC\ &\oH}$ considered as left $\tA$-module. One can check that it is a tilting module for $\tA$
  and
  \[
  ( \End_{\cD(\tA)}T')\op\simeq\tB= \mtr{\qA&\qH/\qA\\0&\oA},  
  \]
  hence $\tA$ is derived equivalent to $\tB$. Unfortunately, this ring can be not so good from the homological point of view. At least, it is not better than
  $\qA$ itself. Namely, as one can easily check,
    $$\gdim\tB=\max(\gdim\qA,1+\pdim_\qA(\qH/\qA)),$$
  which is either $\gdim\qA$ or (more often) $\gdim\qA+1$.

  \medskip
  \item One more observation. Consider the right $\tA$-modules $(\oA\ 0)$ and $(\qC\ \qH)$. One can check that $T''=(\oA\ 0)[1]\+(\qC\ \qH)$ is
  a tilting complex for $\cD(\tA\op)$ and 
  \[
  \End_{\cD(\tA\op)}T''\simeq\qB''=\mtr{\oA&0\\\oH&\qH},
  \]
  hence $\tA\op$ is derived equivalent to $(\qB'')\op$.
  
  Note that the functor $P\mapsto\Hom_\qR(P,\qR)$ induces an exact duality $\Per(\qR)\to\Per(\qR\op)$ for any ring $\qR$. Hence 
 $\Per(\tA)\simeq\Per(\qB'')$.
  \end{enumerate}
  \end{remk}
  
   \begin{exam}\label{e66} 
   Let $\qH=\sT(3,\aK)$ and $\qA=\{ (a_{ij})\in\qH \mid a_{11}=a_{22}\}$. Set $H_i=\qH e_{ii}$ and $U_i=H_i/\rad H_i$. 
   Then $\qC=\{ (a_{ij})\in\qH \mid a_{11}=a_{22}=0\}$, hence $\oH=U_1\+U_2$. On the other hand, $\qC=\rad H_2\+H_3\simeq H_1\+H_3$,
   so $\qC^{-1}=\Hom_\qH(\qC,\qH)$ can be identified with the set of $3\xx2$ matrices $(b_{ij})$ such that $b_{12}=b_{22}=0$.
   One can check that ${_\qC\qH}$ is identified with the subset $\{(b_{ij})\mid b_{11}=0\}\sb\qC^{-1}$ and
   $\oH'\simeq U_2\not\simeq \oH$ (even $\dim_\aK\oH'\ne\dim_\aK\oH$).
   \end{exam}

\end{document}